\documentclass{article}

\usepackage[T1]{fontenc}

\usepackage{amsmath}
\usepackage{amsfonts}
\usepackage{amssymb}
\usepackage{amsthm}

\usepackage{graphicx}

\usepackage{cleveref}

\usepackage{titlesec}

\usepackage{pifont}

\title{On ideal filtrations for Newton nondegenerate surface singularities}
\author{Baldur Sigur\dh sson}
\date{\today}

\newtheorem{thm}{Theorem}[section]
\crefname{thm}{theorem}{theorems}
\Crefname{thm}{Theorem}{Theorems}

\newtheorem{lemma}[thm]{Lemma}
\crefname{lemma}{lemma}{lemmas}
\Crefname{lemma}{Lemma}{Lemmas}

\newtheorem{prop}[thm]{Proposition}
\crefname{prop}{proposition}{propositions}
\Crefname{prop}{Proposition}{Propositions}

\crefname{cor}{corollary}{corollaries}
\Crefname{cor}{Corollary}{Corollaries}

\theoremstyle{definition}

\newtheorem{rem}[thm]{Remark}
\crefname{rem}{remark}{remarks}
\Crefname{rem}{Remark}{Remarks}

\crefname{example}{example}{examples}
\Crefname{example}{Example}{Examples}

\newtheorem{definition}[thm]{Definition}
\crefname{definition}{definition}{definitions}
\Crefname{definition}{Definition}{Definitions}

\crefname{warning}{warning}{warning}
\Crefname{warning}{Warning}{Warning}

\crefname{alg}{algorithm}{algorithms}
\Crefname{alg}{Algorithm}{Algorithms}

\newtheorem*{ack}{Acknowledgements}

\newcommand{\C}{\mathbb{C}}
\newcommand{\N}{\mathbb{N}}
\newcommand{\Z}{\mathbb{Z}}
\newcommand{\R}{\mathbb{R}}
\newcommand{\Q}{\mathbb{Q}}

\newcommand{\V}{\mathcal{V}}

\newcommand{\E}{\mathcal{E}}
\newcommand{\Nd}{\mathcal{N}}

\newcommand{\F}{\mathcal{F}}
\newcommand{\G}{\mathcal{G}}
\newcommand{\I}{\mathcal{I}}
\newcommand{\Gh}{{\hat{\mathcal{G}}}}

\newcommand{\SR}{\mathcal{S}_{\R}}

\newcommand{\X}{\tilde X}

\renewcommand{\O}{\mathcal{O}}

\newcommand{\supp}{\mathop{\rm supp}\nolimits}
\newcommand{\Hom}{\mathop{\rm Hom}\nolimits}
\newcommand{\wt}{\mathop{\rm wt}\nolimits}
\newcommand{\wth}{\mathop{\hat{\rm wt}}\nolimits}
\renewcommand{\div}{\mathop{\rm div}\nolimits}

\newcommand{\convx}{\mathop{\rm conv}\nolimits}

\newcommand{\set}[2]{\left\{ #1 \,\middle\vert\, #2 \right\}}

\newcommand{\fa}[2]{\forall #1 :\, #2}

\begin{document}

\maketitle

\begin{center}
{\it
This article is dedicated to Andr\'as N\'emethi
on his 60\textsuperscript{th} birthday}
\end{center}

\begin{abstract}
We compare three naturally occurring multi-indexed
filtrations of ideals on the local ring
of a Newton nondegenerate hypersurface surface singularity with
rational homology sphere, which in many cases are all distinct.
These are the divisorial, the order, and the image filtrations.
These filtrations are indexed by the lattice associated with a toric
partial resolution of the singularity, or equivalently,
the free abelian group generated by the compact facets of the
Newton polyhedron.

We prove that there exists a top dimensional cone
contained in the Lipman cone having the property that the three ideals
indexed by order vectors from this cone coincide.
As a corollary, if a periodic constant can be associated with the
Hilbert series associated with these filtrations on the Lipman cone,
then they coincide.

In some cases, the Poincar\'e series associated with one of these filtrations
has been shown to coincide with a zeta function associated with the
topological type of the singularity. In the end of the article,
we show that this is the case for all three filtrations considered
in the case of a Newton nondegenerate suspension singularity.
As a corollary, in this case,
the zeta function provides a direct method of determining the
Newton diagram from the link.
\end{abstract}

\section{Introduction} \label{s:intro}

Let $(X,0)\subset (\C^3,0)$ be a hypersurface singularity given as the
vanishing set of a function
$f\in\O_{\C^3,0}$ with Newton nondegenerate principal part. Assume further
that the link is a rational homology sphere. Let $\bar G$
be the dual graph to the
compact Newton boundary of $f$. That is, the vertex set $\Nd$ indexes
the compact facets of $\Gamma_+(f)$ so that for $n\in\Nd$ we
have a face $F_n = F_n(f)$, and two vertices are joined by an edge
if and only if the corresponding faces intersect in a segment.
There is a corresponding toric modification of $\C^3$ which yields a
$V$ resolution $\pi:\X\to X$. To each $n\in\Nd$ there corresponds an
irreducible
component of the exceptional $\pi^{-1}(0)$, say $E_n$.
This correspondence is bijective.

For each $n\in\Nd$ we denote by $\div_n$ the valuation on $\O_{X,0}$
associated with the divisor $E_n$. Furthermore, the positive primitive normal
vector to the face $F_n$ provides a valuation $\wth_n$ on $\O_{\C^3,0}$
which induces the order function $\wt_n$ on $\O_{X,0}$ via
\[
  \wt_n(g) = \max \set{\wth_n(h)}{h|_X = g}.
\]
For $g\in \O_{X,0}$ we set $\div g = (\div_n g)_{n\in\Nd}$ and
$\wt g = (\wt_n g)_{n\in\Nd}$
For $k \in \Z^\Nd$ we define
\[
  \F(k) = \left\{ g \in \O_{X,o} \middle| \div g \geq k \right\}, \quad
  \G(k) = \left\{ g \in \O_{X,o} \middle| \wt  g \geq k \right\}.
\]
Similarly, let $\hat \G$ be the divisorial filtration on $\O_{\C^3,0}$
associated with the valuations $\hat \wt_n$, $n\in \Nd$. We define
$\I(k)$ as the image of $\hat\G(k)$ under the natural projection
$\O_{\C^3,0} \to \O_{X,0}$

It follows from these definition that for all $k\in \Z^\Nd$ we have
inclusions
\begin{equation} \label{eq:inclusions}
  \I(k) \subset \G(k) \subset \F(k).
\end{equation}
In general, we may not expect equality here.
In \cite{Lemah}, Lemahieu shows that the $\I$ and $\G$ coincide
if and only if the Newton diagram of $f$ is bi-stellar,
i.e. every pair of compact facets of $\Gamma_+(f)$ shares a point.
In Example 7.6 of \cite{Nem_Poinc}, N\'emethi provides an example
of a Newton nondegenerate singularity whose diagram
contains only two compact faces (in particular, it is bi-stellar)
for which the inclusion $\G \subset \F$ is shown to be proper.

The following theorem is proved in \cref{s:eq}.
\begin{thm} \label{thm:intro}
Let $(X,0)$ be a Newton nondegenerate hypersurface singularity
in $(\C^3,0)$ with a rational homology sphere link.
Then there exists an $|\Nd|$ dimensional polyhedral cone
$C\subset \SR$
(see \cref{def:cone,def:lipman_cone} for $C$ and $\SR$)
satisfying
\[
  \fa{k\in C\cap \Z^\Nd}
     {\F(k) = \G(k) = \I(k)}.
\]
\end{thm}

In \cref{s:susp} we define the zeta function and prove the following

\begin{thm} \label{thm:susp}
If $(X,0)$ is a Newton nondegenerate suspension singularity
with rational homology sphere link, then $\I, \G, \F$ all coincide.
Furthermore, the associated Poincar\'e series coincides with the reduced zeta
function $Z^\Nd_0(t)$ with respect to nodes
(see def. \cref{def:zeta}), which is given by the formula
\begin{equation} \label{eq:susp_poinc}
  \frac{1-t^{\wth f}}{(1-t^{\wth x})(1-t^{\wth y})(1-t^{\wth z})}.
\end{equation}
\end{thm}

\begin{ack}
I discovered the theorems proved in this article during the PhD
program at Central European University under the supervision of
Andr\'as N\'emethi. I would like to thank Andr\'as for the many
fruitful discussions we have had, and for suggesting to me many
interesting proeblems related to singlarity theory.
\end{ack}

\section{Associated power series and the search for an equation} \label{s:power}
For a better understanding of these filtrations, 
the associated \emph{Hilbert} and \emph{Poincar\'e} series 
are introduced:
\[
  H^\F(t) = \sum_{k\in\Z^\Nd} h^\F_k t^k,\quad
  P^\F(t) = \sum_{k\in\Z^\Nd} p^\F_k t^k
          = -H^\F(t) \prod_{n\in\Nd} (1-t_n^{-1}),
\]
where $h^\F_k = \dim_\C \O_{X,0} / \F(k)$.
Similar definitions are made for the other filtrations.

These series provide very strong numerical invariants of the analytic structure
of the singularity. Two leading questions in the theory of surface
singularities are, on one hand, whether numerical analytic invariants 
such as these can be characterized by the topology of $(X,0)$, and on
the other, whether numerical invariants
can be used to construct
variables and equations realizing singularities with a given topology.

The divisorial filtration $\F$ is intrinsic to the singularity $(X,0)$, and
therefore one may hope for it to have the most direct relation to the link,
whether or not the singularity $(X,0)$ is a hypersurface.
Indeed, in \cite{Nem_Poinc}, N\'emethi provides a topological invariant, the
zeta function, which coincides with $P^\F$
in many cases, e.g. for rational singularities and
minimally elliptic singularities whose minimal resolution is good.
These are examples of classes of singularities whose intrinsic analytic
structure has restrictions. The main identity in \cite{Nem_Poinc}
is not true
for arbitrary singularities, but has been proved for singularities
of splice-quotient type \cite{Nem_coh_splice}.

On the other hand, the filtrations $\I$ and $\G$ are given in terms of
the embedding of the singularity $(X,0) \subset (\C^3,0)$.
It is not clear how to relate the topology of $(X,0)$, or its embedded
type to the Hilbert or Poincar\'e series associated with these filtration.
On the other hand, as we shall see, there are cases when the knowledge of
the Poincar\'e series can be used to rebuild the singularity, or a similar one.

There are no relations between the monomials of the ring
$\O_{\C^3,0}$, and the filtration $\Gh$ is given by
a grading of these monomials.
As a result, one computes easily
(see also Proposition 1 of \cite{Eb_Gus_multi}):
\[
  P^{\Gh}(t) = \frac{1}{(1-t^{\wth x})(1-t^{\wth y})(1-t^{\wth z})}.
\]
By a result of Lemahieu \cite{Lemah}, this gives
\[
  P^\I(t) = (1-t^{\wth f}) P^\Gh
          = \frac{1-t^{\wth f}}{(1-t^{\wth x})(1-t^{\wth y})(1-t^{\wth z})}.
\]
If we assume that $f$ has a \emph{convenient} Newton diagram (meaning in
our case that
$f(x,y,z)$ contains monomials of the form $x^a$, $y^b$, $z^c$ with nonzero
coefficients), then
the arguments of section 5 of \cite{Lemah} show that this series in fact
determines the Newton polyhedron (it is also determined by it).
In particular, if this series can be computed using only the topological
type of $(X,0)$, then one obtains a method of determining from
only the topology of $(X,0)$ an equation for a singularity with that
topological type.
We shall see in \cref{s:susp} that this program actually runs in
the case of suspension singularities with rational homology sphere.

In fact, in \cite{Newton_nondeg},
Braun and N\'emethi found, using totally different methods,
that when the link of a Newton nondegenerate
hypersurface singularity is a rational homology sphere, then the link
determines the Newton diagram, up to permutation of the coordinates.
Nonetheless, the above route identifies a more conceptual way
of finding an equation determining a given topology.

\section{Newton nondegeneracy}

In this section we define the Newton polyhedron and its normal fan.
We do not subdivide the normal fan to obtain
a smooth variety. As a result, we obtain a partial resolution of
$(X,0)$ which has at most cyclic quotient singularities.
This construction is described in details in \cite{Oka}.

Let $f$ be a convergent power series in three variables
given as $f(x) = \sum_{u\in \N^3} a_u x^u$. We define the \emph{support}
of $f$ as
\[
  \supp(f) = \set{u\in \N^3}{a_u \neq 0}
\]
and the \emph{Newton polyhedron} of $f$ as
\[
  \Gamma_+(f) = \convx(\supp(f) + \R^3_{\geq 0}).
\]
A \emph{facet} of $\Gamma_+(f)$ is a face of dimension $2$.
We index the compact facets of $\Gamma_+(f)$ by a set $\Nd$, which
we take as the vertex set of a graph $\bar G$ as in the introduction.
We define the graph $\bar G^*$ similarly, but we allow in this case noncompact
facets as well. We denote the vertex set of $\bar G^*$ by $\Nd^*$.

To a vertex $n\in \Nd^*$, there corresponds a facet
$F_n \subset \Gamma_+(f)$. To each such $n$ there corresponds a unique
primitive integral linear functional $\ell_n:\R^n \to \R$ having
$F_n$ as its minimal set in $\Gamma_+(f)$.

We identify the set of integral linear functionals $\ell:\Z^3\to \Z$
taking nonnegative values on $\N^3$ with $\N^3$ via the standard intersection
product. Thus, for each $n\in \Nd$, the functional $\ell_n$ corresponds to
the primitive normal vector to $F_n$ pointing into $\Gamma_+(f)$.
For any face $F\subset \Gamma_+(f)$ (of any dimension) denote by
\[
  f_F = \sum\set{a_u x^u}{u \in F \cap \supp(f)}.
\]

\begin{definition}
The function $f$ is \emph{Newton nondegenerate} if for any
compact face $F \subset \Gamma_+(f)$, the affine scheme
\[
  \set{x\in (\C^*)^3}{f_F(x) = 0}
\]
is smooth.
\end{definition}

\begin{definition}
The \emph{normal fan}, denoted by $\triangle_f$
of the polyhedron $\Gamma_+(f)$ subdivides the
positive octant $\R_{\geq 0}$ as follows.
\begin{itemize}

\item
The one dimensional cones are generated by $\ell_n$ for $n\in \Nd^*$. 

\item
A two dimensional cone in the normal fan is generated by two vectors
$\ell_n$ and $\ell_{n'}$ where $n,n'$ are adjacent in $\bar G^*$.
Equivalently, for any segment $S = F_n \cap F_n'$, with $\dim S = 1$,
there is a cone consisting of those functionals whose minimal value
on $\Gamma_+(f)$ is taken on all of $S$.

\item
The above construction splits the positive octant into chambers, whose
closures are the three dimensional cones in the normal fan.
Equivalently, to each vertex $u\in \Gamma_+(f)$, there is a three dimensional
cone in the normal fan consisting of those linear functions whose minimum
on $\Gamma_+(f)$ is realized at the point $u$.

\end{itemize}

Denote by $Y_f$ the toric variety associated with $\triangle_f$.
Then we have a canonical morphism $Y_f \to \C^3$. 
Denote by $\bar X \subset Y_f$ the strict transform of $X$.
Denote by $O_n$ the orbit in $Y_f$ corresponding to the cone
generated by $\ell_n$, and by $E_n$ the closure of $O_n\cap \bar X$.
\end{definition}

\section{The intersection lattice}

If $f$ is Newton nondegenerate, then the strict transform $\bar X$
has transverse
intersections with all orbits in $Y_f$, meaning that, if $O$ is an orbit,
then the scheme theoretic intersection $\bar X \cap O$ is smooth.
Furthermore, the divisors $E_n$ are irreducible \cite{Oka}.

We will identify the lattice $\bar L = Z^\Nd$ with the set of divisors on
$\bar X$ supported on the exceptional divisor, that is, the abelian
group freely generated by the irreducible divisors $E_n$ for $n\in \Nd$.
An intersection product is obtained on this lattice as follows. Take
a resolution $\phi:\X\to\bar X$ which is an isomorphism outside the singular
set $\bar X_{\mathrm{sing}}$.
In particular, there is a well defined intersection theory
on $\X$.
For any curve $C \subset \bar X$, the pullback $\phi^*E$ is
defined as $\tilde C + C_{\mathrm{exc}}$,
where $\tilde C$ is the strict transform of
$C$, and $E$ is the unique rational divisor supported on
$\phi^{-1}(\bar X_{\mathrm{sing}})$,
satisfying $(E,C_{\mathrm{exc}}) = 0$ for any divisor
$E$ supported on $\phi^{-1}(\bar X_{\mathrm{sing}})$.
We then set $(C,C') = (\phi^*C, \phi^*C')$.

\begin{definition} \label{def:t_alpha}
We refer to $\bar L$ with the intersection form defined above as the
\emph{intersection lattice}. Elements of $\bar L$, or or
$\bar L_\R = \bar L \otimes \R$
are referred to as \emph{cycles}.
Let $n \in \Nd$ and $n'\in\Nd^*$.
We set $e_n = E_n^2 = (E_n, E_n)$. Furthermore
\begin{itemize}

\item
Denote by $t_{n,n'}$
the length of the segment $F_n \cap F_{n'}$, that is, the
number of relative interior integral points on this segment.
In particular, $t_{n,n'} = 0$ if and only if $n,n'$ are not adjacent.

\item
Denote by $\alpha_{n,n'}$ the index of the lattice generated by
$\ell_n$ and $\ell_{n'}$ in its saturation in $\Hom(\bar L, \Z)$.

\end{itemize}
\end{definition}

\begin{prop} \label{prop:intersection}
The intersection lattice is negative definite. In particular, we
have $e_n < 0$.
Let $n,n' \in \Nd$ be adjacent in $\bar G$. Then
$(E_n, E_{n'}) = t_{n,n'}/\alpha_{n,n'}$.
Furthermore, for any $n\in \Nd$, we have
\[
  e_n \ell_n + \sum_{n'} \frac{t_{n,n'}}{\alpha_{n,n'}} \ell_{n'} = 0.
\]
\end{prop}
\begin{proof}
The intersection lattice can be seen as a subspace of the intersection
lattice associated with a resolution of $(X,0)$, which is negative
definite, see e.g. \cite{Nemethi_FL}.
The rest follows from \cite{Oka}, see also \cite{Newton_nondeg}.
\end{proof}

\section{Cycles, Newton diagrams and the cone}

In this section we define the cone $C$ which appears in
\cref{thm:intro}. This requires some analysis of the geometry of
Newton diagarams associated with arbitrary cycles.
\Cref{lem:cone} shows that $C$ has the right properties, that is,
it is a top dimensional rational cone contained in the Lipman cone.
\Cref{lem:containment} is a workhorse used in the proof
of \cref{thm:intro}.

\begin{definition} \label{def:lipman_cone}
The \emph{Lipman cone} $\SR$ is the set of vectors $Z\in \bar L_\R$ satisfying
$(Z, E) \leq 0$ for any effective cycle $E$.
\end{definition}

It is well known that the Lipman cone is an $|\Nd|$-dimensional simplicial
cone generated by elements with all coordinates positive.

We associate to a cycle $Z\in \bar L_\R$ the \emph{Newton polyhedron}
\[
  \Gamma_+(Z) = \left\{ u \in \R_{\geq0}^3
             \middle| \forall n\in\Nd,\, \ell_n(u) \geq m_n(Z) \right\}
\]
where the $m_n$ are defined by $Z = \sum_{n\in\Nd} m_n(Z) E_n$.
For a subgraph $A$ of $\bar G$ (or a subset of $\Nd$)
let $\Nd_A$ be the set of vertices either in $A$ or
connected to a vertex in $A$. For a cycle $Z$ let
\[
  \Gamma^A_+(Z) = \left\{ u \in \R_{\geq0}^3
             \middle| \forall n\in \Nd_A,\, \ell_n(u) \geq m_n(Z) \right\}
\]
and for $a\in A$, denote by $F^A_a(Z)$ the corresponding face of this
polyhedron, given by
\[
  F^A_a(Z) = \set{u \in \Gamma_+^A(Z)}{\ell_a(u) = m_n(Z)}.
\]
Note that we may have $F^A_a(Z) = \emptyset$.

\begin{definition} \label{def:cone}
Let $C$ be the set of divisors $Z\in \bar L$ satisfying
\begin{itemize}
\item
$\emptyset \neq F_n^{\{n\}}(Z) = F_n(Z)$ for all $n\in\Nd$.
\item
If $n,n'\in\Nd$ are adjacent in $\bar G$
and $\rho (F_n(f) \cap F_{n'}(f)) + u \subset F_n(Z)\cap F_{n'}(Z)$ 
for some $\rho \geq 0$ and $u\in\R^3$
then $\rho F_n(f) + u \subset F_n(Z)$.
\end{itemize}
\end{definition}

\begin{lemma} \label{lem:cone}
$C$ is a top dimensional polyhedral cone contained in the Lipman cone
$\SR$.
\end{lemma}

\begin{proof}
The definition of $C$ is equivalent to a finite number of
rational inequalities, and so the set $C$ is a rational polyhedron.
Furthermore, assume that $\lambda\in \R_{\geq 0}$ and $Z,Z'\in C$.
Then $F_n^A(\lambda Z) = \lambda F_n^A(Z)$ for any $A\subset \Nd$,
which shows $\lambda Z \in C$.
Furthermore, $F_n^A(Z+Z') = F_n^A(Z) + F_n^A(Z')$. Thus,
if $n,n'\in \Nd$ are adjacent in $\bar G$, and
\[
  \rho>0, \quad u\in \R^3, \quad
  \rho(F_n(f)\cap F_{n'}(f)) + u \subset F_n(Z+Z')\cap F_{n'}(Z+Z'),
\]
then there are $\rho_1, \rho_2 > 0$, $u_1, u_2 \in \R^3$ so that
\[
\begin{split}
  \rho_1(F_n(f)\cap F_{n'}(f)) + u_1 &\subset F_n(Z)\cap F_{n'}(Z),\\
  \rho_2(F_n(f)\cap F_{n'}(f)) + u_2 &\subset F_n(Z')\cap F_{n'}(Z'),
\end{split}
\]
and we get
\[
  \rho F_n(f) + u = (\rho_1 F_n(f) + u_1) + (\rho_2 F_n(f) + u_2)
  \subset F_n(Z) + F_{n'}(Z) = F_n(Z+Z').
\]
As a result, we find $Z,Z' \in C$, and so $C$ is a cone.

Next, we prove $C\subset \SR$. Let $n\in \Nd$ and choose an
$u \in F_n^{\{n\}}$, which is nonempty by assumption. We find
\[
  (E_n, Z) = e_n m_n(Z) + \sum_{n'\in\Nd_n} \frac{t_{n,n'} m_{n'}(Z)}
                                                  {\alpha_{n,n'}}
    \leq e_n \ell_n(u) + \sum_{n'\in\Nd_n} \frac{t_{n,n'} \ell_{n'}(u)}
                                                 {\alpha_{n,n'}}
    = 0.
\]

Finally, we prove that $C$ has dimension $|\Nd|$.
We will use the terminology introduced in \cite{Newton_nondeg}, in
particular, central faces and edges, arms and hands.
Let $n_0 \in\Nd$ be a vertex so that $F_{n_0}(f)$ intersects all the
coordinate planes. Then the complement $\Nd\setminus n_0$ is a disjoint
union of parts of arms.
Let the vertices of the $k$-th partial arm have vertices
$n_{k,j}$ in such a way that $n_{k,1}$ is adjacent to $n_0$, and
for $j\geq 2$, $n_{k,j}$ is adjacent to $n_{k,j-1}$.
We also set $n_{k,0} = n_0$ for any $k$.

Define $Z\in \bar L_\R$ recursively as follows.
Start by choosing $\varepsilon>0$ very small
and set $m_{n_0}(Z) = \wth_{n_0} f$ and
$m_{n_{k,1}}(Z) = \wth_{n_{k,1}}(f) - \varepsilon$.
Note that at this point we have a well defined facet
\[
  F_{n_0}^{\{n_0\}}(Z)
  = \set{u\in\R_{\geq 0}^3}
        {
           \ell_{n_0} = m_{n_0}(Z),\quad
           \fa{k}{\ell_{n_{k,1}} \geq m_{n_{k,1}}(Z)}
        }
\]
and it follows from this construction that this face intersects each
coordinate hyperplane in a segment of positive length.

Next, assume that we have defined $m_{n_{k,j}}$ for $0<j\leq j_0$
for some $j_0 > 0$.
In particular, the facet $F_{n_{k,j_0-1}}^{\{n_{k,j_0-1}\}}(Z)$
is well defined similarly as above.
Unless $n_{k,j_0}$ is a hand, define
\[
  m_{n_{k,j_0+1}}(Z)
  = \min\set{\ell_{n_{k,j_0+1}}(u)}
            {u\in F_{n_{k,j_0-1}}^{\{n_{k,j_0-1}\}}(Z)} - \varepsilon.
\]
In particular, the face
$F_{n_{k,j_0}}^{\{n_{k,j_0}\}}(Z)$ is now well defined.

Note now that if $n$ is a node, and $F_n^{\{n\}}(Z)$ is already well defined,
then the value $m_{n_{k,j_0+1}}(Z)$ is smaller than the minimal value
of $\ell_{n_{k,j_0+1}}$ on $F_n^{\{n\}}(Z)$. Therefore, we find
\[
  \fa{n\in \Nd}{F_n^{\{n\}}(Z) = F_n(Z)},
\]
proving the first condition for $Z \in C$.
The second condition follows similarly.

Finally note that at every step in the definition of $Z$, we may as well
have used a differt epsilons, meaning that a generic small perturbation
of $Z$ is also in $C$. It follows that $C$ contains an open subset
of $\bar L_\R$, and so has highest dimension possible, $|\Nd|$.
\end{proof}

\begin{rem} \label{rem:nonzero}
By the above lemma, if $Z \in C$, then either $Z = 0$, or all coordinates
of $Z$ are positive, that is, $m_n(Z) > 0$ for all $n\in \Nd$, since
this holds for any element of the Lipman cone.
\end{rem}

\begin{lemma} \label{lem:containment}
Let $Z\in C$, $\rho\in\R_{>0}$ and $u\in\R^3$ satisfying
$\rho F_n(f) + u \subset F_n(Z)$ for some $n\in\Nd$. Then
$\rho \Gamma_+(f) + u \subset \Gamma_+(Z)$.
\end{lemma}

\begin{proof}
For $A\subset\Nd$ a subset inducing a connected subgraph of $\bar G$ containing
$n$, let $P_{A}(Z)$ be the following condition:
\begin{enumerate}
\item \label{cond:contains}
We have $\rho F_k(f) + u \subset F^{A}_k(Z)$ for all $k\in A$.
\item \label{cond:behaves}
For any $l\in\Nd \setminus A$,
$l'\in\Nd_l$ and dilation $\phi:\R^3\to\R^3$, $x\to \rho'x + u'$
so that
$\phi(F_l(f)\cap F_{l'}(f)) \subset F^B_l(Z)\cap F^B_{l'}(Z)$
where $B$ is the connected component of $\bar G\setminus A$ containing $l$,
we have $\phi(F_l(f))\subset F^B_l(Z)$.
\end{enumerate}

The assumptions of the lemma imply $P_{\{n\}}(Z)$.
Assuming there is a $Z'\in \bar L$ with $Z'\geq Z$
so that $P_\Nd(Z')$ holds, we find
$\rho\Gamma_+(f)+u\subset \Gamma_+(Z') \subset \Gamma_+(Z)$, proving
the lemma. Thus, it is enough to prove that given an $n\in A\subset\Nd$
inducing a connected subgraph of $\bar G$, and
a $Z'\geq Z$ so that $P_{A}(Z')$ holds,
and an $i\in\Nd_A\setminus A$,
there is a $Z''\geq Z'$ so that $P_{A\cup \{i\}}(Z'')$ holds.

So, let such an $i$ be given, assume that it is adjacent
in $\bar G$ to a $j\in A$. Since $\rho F_j(f) + u \subset F^A_j(Z)$ we have
$m_i(Z) \leq \rho \wth_i(f) + \ell_i(u)$. Let
$s = \frac{\rho \wth_i(f) + \ell_i(u)}{m_i(Z)}$.
Note that the denominator here is nonzero by \cref{rem:nonzero}.
Then $s\geq 1$.
Let $B$ be the connected component of $\bar G \setminus A$ containing
$i$ and define the cycle $Z''$ by
\[
  m_k(Z'') =
  \begin{cases}
    s m_k(Z) & \textrm{if } k \in B, \\
      m_k(Z) & \textrm{else}. \\
  \end{cases}
\]
Then $Z''\geq Z'$. We start by noting that condition
$P_{A\cup \{i\}}(Z'')$\ref{cond:behaves} follows immediately from
$P_{A}(Z')$\ref{cond:behaves}.

We are left with proving
$P_{A\cup \{i\}}(Z'')$\ref{cond:contains}.
We must show that for $k\in A\cup \{i\}$ and $l\in \Nd_{A\cup \{i\}}$ we have
\begin{equation} \label{eq:pf_i}
  m_l(Z'') \leq \min_{\rho F_k(f)+u} \ell_l,
\end{equation}
with equiality in the case $k=l$.

If $k\in A$ and $l\neq i$,
then this is clear from $P_A(Z')$\ref{cond:contains}.

The minimum of $\ell_i$ on $\cup_{k\in A} \rho F_k + u$ is taken on
$(\rho F_i + u) \cap (\rho F_j + u)$, and so by definition of
$m_i(Z'')$, \cref{eq:pf_i} holds also for $l=i$ and any $k\in A$.

\Cref{eq:pf_i} is also clear when $k = i$ and $l$ is either $i$ or $j$.

Finally, we prove \cref{eq:pf_i} in
the case $k \in A\cup \{i\}$ and $l \neq j$.
Similarly as above, the function $\ell_l$ restricted to
$\cup_{k\in A\cup\{i\}} \rho F_k + u$ takes its minimal value on
$(\rho F_i + u) \cap (\rho F_l + u)$, and so it suffices to consider
the case $k=i$. 

Let $\rho'>0$ and $u'\in \R^3$ be such that
$\rho'(F_i(f) \cap F_j(f)) + u' = F_i(Z') \cap F_j(Z')$.
By $P_A(Z')$\ref{cond:behaves}, we have $\rho' F_i(f) + u' \subset F_i(Z')$.
By the definition of $Z''$, we find $s\cdot F_i(Z') \subset F_i(Z'')$.
As a result,
\[
  s(\rho' F_i(f) + u') \subset s F_i(Z') \subset F_i(Z'').
\]
An application of \cref{lem:homotheties} now shows that if
$\rho'' > 0$ and $u''$ are such that
$\rho''(F_i(f) \cap F_j(f)) + u'' = F_i(Z'') \cap F_j(Z'')$, then
$\rho''F_i(f) + u'' \subset F_i(Z'')$.
Now, we get
\[
  \rho(F_i(f) \cap F_j(f)) + u \subset \rho''(F_i(f) \cap F_j(f)) + u''.
\]
which then implies
\[
  \rho F_i(f) + u \subset \rho'' F_i(f) + u'' \subset F_i(Z''),
\]
which is $P_{A\cup\{i\}}(Z'')$\ref{cond:contains} for $k=i$.
\end{proof}

\begin{lemma} \label{lem:homotheties}
Let $A \cong \R^2$ be an affine plane,
$\ell_i:A\to \R$ affine functions for $i=0,\ldots, s$
Assume that $P,Q\subset A$ are polygons given by inequalities
$\ell_i \geq p_i$ and $\ell_i\geq q_i$ respectively, in such a way that
$p_i = \min_P \ell_i$ and $q_i = \min_Q \ell_i$.
Let $P_i$ and $Q_i$ be the minimal sets of $\ell_i$ on $P$ and $Q$
respectively.
We assume that $Q \subset P$ and that $Q_0 = P_0$ is a segment
of positive length.

Take a $p'_0 < p_0$ in such a way that we have a polygon $P'$ defined
by inequalities $\ell_0 \leq p_0'$ and $\ell_i \leq p_i$ for $i>0$, and
$p'_0 = \min_{P'}\ell_0$, and define $P_i'$ as the minimal set of $\ell_i$
on $P'$. Assume that $P_0'$ is a segment of positive length.
Let $\phi:A\to A$ be the unique affine isomorphism which preserves directions
(i.e. if $L\subset A$ is a line, then $L$ and $\phi(L)$ are parallel)
so that $\phi(P_0) = P'_0$. Then $\phi(Q) \subset P'$.
\end{lemma}

\begin{proof}
We can assume that $P'_1$ and $P'_s$ are adjacent to $P'_0$. Consider
three cases.

The first case is when the lines spanned by
the segments $P'_1$ and $P'_s$ are not parallel,
and their intersection point $a$ satisfies
$\ell_0(a) < p_0$. In this case, $\phi$ is a homothety with center
$a$ and ratio $< 1$. As a result, if we define
$P^1$ as the convex hull of $P$ and $a$, then $\phi(P) \subset P^1$.
In particular, $\phi(Q) \subset P_1$.
The polygon $P^1$ can be defined by the inequalities
$\ell_i \geq c_i$ for $i>0$. It is clear that $\phi(Q)$ also satisfies
$\ell_0 \geq c_0'$. As a result, $\phi(Q) \subset P'$.

\begin{figure}[ht]
\begin{center}
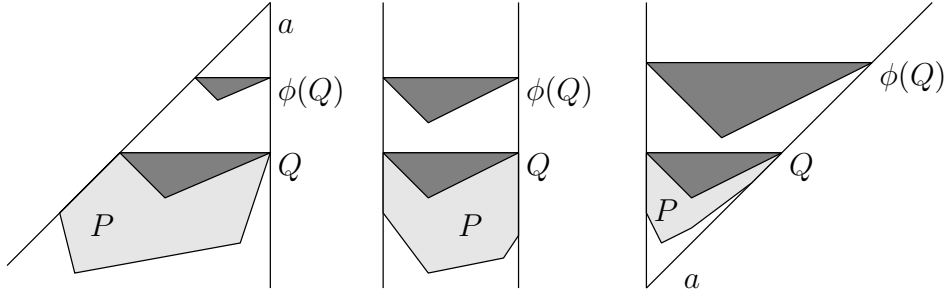
\caption{Two homotheties and a translation.}
\label{fig:homotheties}
\end{center}
\end{figure}

In the second case, assume that the segments $P'_1$ and $P'_s$ are
parallel. In this case, $\phi$ is a translation preserving the lines
spanned by $P'_1$ and $P'_s$, and $P'$ is the convex hull of
$P$ and $\phi(P)$. In particular, $Q \subset P'$.

In the third case, the lines spanned by $P'_1$ and $P'_s$ are not
parallel, and their intersection point $a$ satisfies
$\ell_0(a) > c_0$. In this case, $\phi$ is a homothety with center
$a$ and ratio $>1$ and similarly as in the second case,
$P'$ is the convex hull of $P$ and $\phi(P)$, and so $\phi(Q) \subset P'$.
\end{proof}

\section{Equality between ideals} \label{s:eq}
For $g\in\O_{\C^3,0}$ denote by $g_n$ the principal part of $g$ with respect
to the weight function $\ell_n$.
For $i=1,2,3$ and $\hat g\in \O_{\C^3,0}$,
we denote by $\wth_i$ the weight of $g$ with respect to the $i$-th
natural basis vector, i.e. $\wth_i(x_j) = \delta_{i,j}$.

\begin{lemma} \label{lem:EGZ}
Let $g\in\O_{\C^3,0}$. Then $\wth_n g \leq \div_n g|_X$ with sharp inequality
if and only if $f_n$ divides $g_n$ over the ring of Laurent polynomials.
\end{lemma}

\begin{proof}
See e.g. the proof of Proposition 1 of \cite{Ebeling_Gusein-Zade}.
\end{proof}

\begin{lemma} \label{lem:axis}
Let $g \in \O_{\C^3,0}$ and assume $\wth_n g < \div_n g$ for some $n\in\Nd$.
Let $h = g_n / f_n$ (a Laurent polynomial by \ref{lem:EGZ}).
Writing $\{1,2,3\} = \{i,j,k\}$, if 
$F_n(f)$ intersects the $x_jx_k$ coordinate plane, then $\wth_i(h) \geq 0$.
\end{lemma}

\begin{proof}
Assume that $h$ contains a monomial with a negative power of $x_i$.
Then the same would hold for $g_n = h f_n$, since $f_n$ contains monomials with
no power of $x_i$.
\end{proof}

\begin{proof}[Proof of \cref{thm:intro}]
We want to show that for any $Z \in C$, we have $\F(Z) = \G(Z) = \I(Z)$.
In light of \cref{eq:inclusions}, it suffices to show that $\F(Z)$
contains $\I(Z)$, that is, if $g \in \F(Z)$, then there exists a
$\hat g\in \O_{\C^3,0}$ restricting to $g$
with $\wth \hat g \geq m_n(Z)$ for all $n\in \Nd$.

We use the classification in \cite{Newton_nondeg} to set up an induction
on the vertices of $\bar G$.
Assume that $n_0$ is a vertex which intersects all the coordinate axis.
This can be done by Proposition 2.3.9 of \cite{Newton_nondeg}
by choosing $F_{n_0}$ either as a central facet
or containing a central edge.
We define the partial ordering $\leq$ on $\Nd$ by setting
$n_1 \leq n_2$ if $n_1$ lies on the geodesic connecting $n_0$ and $n_2$.
Note that $\bar G$ has well defined geodesics since it is a tree.

We prove inductively the statement $P(A)$ that for a subset
$A\subset \Nd$ satisfying
\[
  n\in A,\quad n'\leq n,\quad\Rightarrow\quad n'\in A,
\]
there exists
a $\hat g\in \O_{\C^3,0}$ satisfying $\hat g|_X = g$ and
$\wth_n g \geq m_n(Z)$ for any $n\in A$.

The initial case $P(\emptyset)$ is clear, but we prove
$P(\{n_0\})$ as well. Take any $\hat g \in \O_{\C^3,0}$ restricting
to $g$. If $\wt_{n_0} \hat g < m_{n_0}(Z)$, then by \cref{lem:EGZ}
there is a Laurent polynomial $h$ so that
$\wth_{n_0} (\hat g - hf) > \wth_{n_0} \hat g$. By our choice of $n_0$
and \cref{lem:axis}, $h$ is a polynomial, and so we can replace
$\hat g$ with $\hat g - hf \in \O_{\C^3,0}$. After repeating this argument
finitely many times, we can assume that $\wth_{n_0} \hat g \geq m_{n_0}(Z)$.

Next, assume that $A \subset \Nd$ satisfies our inductive hypothesis,
and that $n\in \Nd$ is a minimal element of $\Nd\setminus A$.
It suffices to find a
polynomial $h$ such that $\hat g - hf$ $P(A)$, as well as
$\wth_n (\hat g-hf) > \wth_n(\hat g)$.

By \ref{lem:EGZ} there does exist a Laurent polynomial $h$ so that
$\wth_n(\hat g-hf) > \wth_n \hat g$. Indeed, set $h = \hat g_n / f_n$.
We can assume that $F_n(f)$ intersects the $x_1x_3$ and $x_2x_3$ coordinate
hyperplanes. 
By \ref{lem:axis} we have $\wth_1 h \geq 0$ and $\wth_2 h \geq 0$.
In order to finish the proof, it therefore suffices to show
$\wth_3(h) \geq 0$ and $\wt_a(hf) \geq m_n(Z)$.

We construct a cycle $Z'$ as follows.
Let $a$ be the unique vertex in $A$ adjacent to $n$ and
$p$ the unique point on the $x_3$ axis satisfying $\ell_a(p) = m_a(Z)$.
Set $m_k(Z') = m_k(Z)$ for all
$k$ in the connected component of $\bar G\setminus n$ containing $A$,
otherwise set $m_k(Z') = \ell_k(p)$. As a result, the Newton polyhedron
$\Gamma_+(Z')$ of $Z'$ is the convex closure of $\Gamma(Z)$ and the point
$p$. In particular, if $k$ is in the connected component of $G\setminus n$
containing $A$, then either $F_k(Z') = F_k(Z)$, or $k=a$ and
$F_a(Z') \subset F_a(Z)$. For any other vertex $k$, we have
$F_k(Z') = \{p\}$. It follows from this that $Z' \in C$.

In fact, we find that
\[
  x\in \R^3_{\geq 0},\quad
  \ell_a(x) = m_a(Z),\quad
  \ell_n(x) \leq m_n(Z)\quad
  \Rightarrow\quad
  x\in F_a(Z').
\]
Now let $u\in \supp(h)$ and $w \in \supp(f_n)$. We then have
$\ell_a(u+w) \geq m_a(Z)$ and $\ell_n(u+w) < m_n(Z)$. Since
$\ell_a(0,0,1) > 0$, there is a $t>0$ so that
$\ell_a(u+w-(0,0,t)) = m_n(Z)$, and we also have
$\ell_n(u+w-(0,0,t)) < m_n(Z)$. We have thus proved that
$F_n(f) + u - (0,0,t) \subset F_n(Z')$. \Cref{lem:containment} now gives
the middle containment in
\[
  \Gamma_+(f) + u
  \subset \Gamma_+(f)+u-(0,0,t)
  \subset \Gamma_+(Z') \subset \R^3_{\geq 0},
\]
which implies, on one hand,
that $\wth_k(hf) \geq m_k(Z') = m_k(Z)$ for all $k\in A$,
and on the other hand, $\wth_3(h) = \wth_3(hf) \geq 0$, finishing the proof.
\end{proof}

\section{Suspension singularities} \label{s:susp}
In this section we consider suspension singularities. In this case,
a stronger statement than \cref{thm:intro} holds, namely, the three filtrations
all coincide. 
Most of the work in this section, however, goes into proving the
\emph{reduced identity} for nodes for suspension singularities,
see \cite{Nem_Poinc} Definition 6.1.5.
This means that the Poincar\'e series associated with
the filtration $\F$ (or $\G$ or $\I$, as they coincide in this case)
is identified by a topological invariant, the \emph{zeta function}
associated with the link of the singularity.

In this section we assume that $(X,0)$ is a suspension singularity, that
is, there is an $f_0\in \O_{\C^2,0}$ and an $N \in \Z_{>1}$
so that $(X,0)$ is given by an equation
$f = 0$, where $f(x,y,z) = f_0(x,y) + z^N$.
Newton nondegeneracy for $f$ means that $f_0$ is Newton nondegenerate.
For convenience, we will also assume that the diagram of $f$
is convenient. This is equivalent to $f_0$ not vanishing along the
$x$ or $y$ axis.

\begin{proof}[Proof of \cref{thm:susp}]
If $f$ is the $N$-th suspension of an equation of a plane curve given by
$f_0 = 0$, so that $f(x,y,z) = f_0(x,y) + z^N$, then every compact facet of
$\Gamma_+(f)$ is the convex hull of a compact facet of the Newton polyhedron
of $f_0$ and the point $(0,0,N)$. In particular, $\Gamma_+(f)$ is
\emph{bi-stellar}, and so by Proposition 4 of \cite{Lemah}, we have
$\I = \G$.

Now, let $n\in \Nd$ correspond to the facet $F_n \subset \Gamma_+(f)$.
By the description above, $F_n$ intersects all coordinate hyperplanes.
If $\hat g\in \O_{\C^3,0}$ and $\wth_n \hat g < \div_n g|_X$, then
by \cref{lem:EGZ,lem:axis}, there is a polynomial $h$ so that
$\wth_n \hat g - fh > \wth \hat g$. As a result, we find
$\wt_n g = \div_n g$ for $g = \hat g|_X$, that is, $\F = \G$.

The formula for the Poincar\'e series is shown in \cref{s:power} to
follow from \cite{Lemah}.
The formula for the zeta function is \cref{thm:zeta}.
\end{proof}

Using a smooth subdivision of the normal fan to $\Gamma_+(f)$, we obtain
an embedded resolution of $(X,0)$, whose resolution graph we denote by
$G$. This graph is obtained as follows.
From $\bar G^*$, construct $G^*$ by replacing edges between
$n,n'\in \Nd$ with a string, and an edge between $n\in \Nd$ and
$n'\in \Nd^*\setminus \Nd$ with $t_{n,n'}$ bamboos.
The graph $G$ is obtained from $G^*$ by removing the vertices in
$\Nd^*\setminus\Nd$, see \cite{Oka} for details.
We denote by $\V$ the vertex set of $G$, and we have a natural inclusion
$\Nd \subset \V$, where if $v\in \V$, then $v\in \Nd$ if and only if
$v$ has degree $>2$. We denote by $\E$ the set of vertices in $G$
with degree $1$. Note that if $e\in \E$, then there are unique
$n\in\Nd$ and $n'\in \Nd^*\setminus\Nd$ so that $e$ lies on a bamboo
connecting $n$ and $n'$. We set $\alpha_e = \alpha_{n,n'}$ in this case,
recall \cref{def:t_alpha}. For a given $n$, we denote the set of such
$e\in \E$ by $\E_n$. Thus, the family $(\E_n)_{n\in\Nd}$ is a partitioning
of $\E$.
A vertex $v\in \V$ corresponds to an irreducible
component of the exceptional divisor $E_v$.

The associated intersection lattice is negative definite, in particular,
the intersection matrix is invertible. Thus, for $v\in \V$,
we have a well defined
cycle $E_v^*$, that is, divisor supported on the exceptional divisor of
the resolution, satisfying $(E^{\phantom{*}}_w, E_v^*) = 0$ if $w\neq 0$, but
$(E^{\phantom{*}}_v, E_v^*) = -1$.
We denote the lattice generated by $E_v$ by $L$, and the lattice
generated by $E_v^*$ by $L'$. We then have
$L = H_2(\X, \Z)$ and $L' = H_2(\X,\partial\X,\Z) = \Hom(L,\Z)$.

Write $\Gamma_+(f_0) = \cup_{i=1}^r\Gamma_0^i$, where
$\Gamma_0^i = [(a_{i-1},b_{i-1}),(a_i,b_i)]$
are the facets of the Newton polyhedron of $f_0$,
so that $0=a_0<\ldots< a_r$ and $b_r = 0$.
Let $s_i$ be the length of the $i$-th segment,
that is, the content of the vector
$(a_i-a_{i-1}, b_i-b_{i-1})$.
Let $F_{n_i}$ be the facet of $\Gamma_+(f)$ containing the segment
$[(a_{i-1}, b_{i-1}), (a_i, b_i)]$.
Furthermore, let $s_x = \gcd(N,b_0)$ and $s_y = \gcd(N,a_r)$.
Then, in fact, if $n_x, n_y$ are the vertices in $\Nd^*$ corresponding
to the $yz$ and $xz$ coordinate hyperplanes, respectively, then
$s_x = t_{n_1, n_x}$ and $s_y = t_{n_r, n_y}$.

It can happen that the diagram $\Gamma(f)$ is not minimal in the
sense of \cite{Newton_nondeg}. This is the case
if $s_x = N$, $s_y = N$, $a_1 = 1$ or $b_{r-1} = 1$. If this is the case, we
blow up the appropriate points to produce redundant legs consisting of a
single $-1$ curve to make sure that nodes, that is, vertices of degree
$>2$ in $G$ correspond to facets
in $\Gamma(f)$ and their legs correspond to primitive segments on
the boundary of $\Gamma(f)$. In particular, we assume that
$\wt xyz = \sum_{e\in\E} E_e^*$.

The sets $\E_{n_1}$ and
$\E_{n_r}$ have special elements $e_j^x$, $1\leq j\leq s_x$, and $e_j^y$,
$1\leq j \leq s_y$, corresponding to the
segments $[(0,b_0,0),(0,0,N)]$ and $[(a_r,0,0),(0,0,N)]$, respectively. 
Set $\E_1^x = \{e^x_i|1\leq i\leq s_x\}$ and $\E_i^x = \emptyset$ for
$i>1$. Similarly, set $\E_r^y = \{e^y_i|1\leq i\leq s_y\}$ and
$\E_i^y = \emptyset$ for $i<r$. Further, let
$\E_i^z = \E_{n_i} \setminus (\E_i^x\cup \E_i^y)$.
Set also $\E^t = \cup_i \E_i^t$ for $t=x,y,z$.
Note that we get $|\E_i^z| = s_i$. Define $s_z = \sum_i s_i$. Write
$\E_i^z = \{e_1^{z,i}, \ldots, e_{s_i}^{z,i}\}$.
Note that the number $\alpha_e$ is constant for $e\in\E^x$ (in fact, we
have $\alpha_e = a_1/s_1$). We denote this
by $\alpha_x$. Define $\alpha_y$ similarly.

If $1<i<r$ we have $\alpha_e = N$ for $e\in\E_i^z$. We have
$\alpha_e = N/s_x$ for $e\in\E_1^z$ and
$\alpha_e = N/s_y$ for $e\in\E_r^z$.

\begin{lemma} \label{lem:ends_nodes}
Let $n\in\Nd$ and $e\in\E_n$. Then $\alpha_e E_e^* - E^*_n \in L$.
Furthermore, $\alpha_e E_e^* - E^*_n$ is supported on the leg containing
$e$, that is, the connected component of $G\setminus n$ containing $e$.
\end{lemma}
\begin{proof}
This follows from Lemma 20.2 of \cite{EisNeu}.
\end{proof}

\begin{definition}
Let $H$ be the first homology group of the link of $(X,0)$. Thus,
$H = L'/L$, where $L\subset L'$ via the intersection product.
If $l \in L'$, we denote its class in $H$ by $[l]$.
\end{definition}

\begin{lemma} \label{lem:H_order}
The order of $H$ is $N^{s_z-1} \alpha_x^{s_x-1} \alpha_y^{s_y-1}$.
\end{lemma}
\begin{proof}
From the proof of Theorem 8.5 of \cite{Milnor_hyp}, we see that in fact,
$|H| = \Delta(1)$, where $\Delta$ is the characteristic polynomial
of the monodromy actiong on the second homology of the Milnor fiber.
We leave to the reader to verify, using \cite{Var_zeta}, that the
characteristic polynomial is, in our case, given by the formula
\[
\begin{split}
  \Delta(t) =
  &\left[ 
    \left(\prod_{i=1}^r (t^{m_i} - 1)^{s_i}\right)
    \left(t^{m_1}-1\right)^{s_x-1}
    \left(t^{m_r}-1\right)^{s_y-1}
  \right]\\
  &\left[
    \left(\prod_{i=1}^r (t^{\frac{m_i}{\alpha_i}} - 1)^{s_i}\right)
    \left(t^{\frac{m_1}{\alpha_x}}-1\right)^{s_x}
    \left(t^{\frac{m_r}{\alpha_y}}-1\right)^{s_y}
  \right]^{-1}\\
  &\left[
    \left(t^{\frac{m_1}{\alpha_1\alpha_x}}-1\right)
    \left(t^{\frac{m_r}{\alpha_r\alpha_y}}-1\right)
    \left(t^N-1\right)
  \right]\\
  &(t-1)^{-1},
\end{split}
\]
where for $i=1,\ldots, r$, we take $m_i\in \Z$ so that the facet $F_{n_i}$ of
$\Gamma_+(f)$ containing $[(a_{i-1}, b_{i-1}), (a_i, b_i)]$ is contained
in the hyperplane $\ell_{n_i} \equiv m_i$.
This implies
\[
\begin{split}
  \Delta(1) =&
  \frac{
    \left[\prod_{i=1} m_i^{s_i}\right]
    m_1^{s_x-1}
    m_r^{s_y-1}
    \left( \frac{m_i}{\alpha_1 \alpha_x} \right)
    \left( \frac{m_r}{\alpha_r \alpha_y} \right)
    N
  }
  {
    \left[\prod_{i=1} \left(\frac{m_i}{\alpha_i}\right)^{s_i}\right]
    \left( \frac{m_i}{\alpha_1} \right)^{s_x}
    \left( \frac{m_r}{\alpha_r} \right)^{s_y}
  }\\
  =&\left[\prod_{i=1}^r \alpha_i^{s_i}\right]
    \alpha_1^{-1} \alpha_r^{-1}
    \alpha_x^{s_x-1} \alpha_y^{s_y-1} N
\end{split}
\]
Now, for $1<i<r$, we have $\alpha_i = N$. Furthermore, if
$s_1 \neq 1$, then $s_x = 1$ and $\alpha_1 = N$. Similarly, if
$s_r \neq 1$, then $s_y = 1$ and $\alpha_r = N$. As a result,
the above product equals $N^{s_z-1} \alpha_x^{s_x-1} \alpha_y^{s_y-1}$.
\end{proof}

\begin{lemma} \label{lem:En_triv}
For $1\leq i \leq r$, let $g_i$ be a generic sum of 
$x^\frac{a_{i+1}-a_i}{s_i}$ and $y^\frac{b_i-b_{i+1}}{s_i}$.
Then, for $1<i<r$ we have $\div g_i = E_{n_i}^*$.
In particular, $[E_{n_i}^*] = 0 \in H$.

Furthermore, we have 
$\div g_1 = s_x E_{n_1}^*$ and $\div g_r = s_y E_{n_r}$). In particular,
$s_x[E_{n_1}^*] = s_y[E_{n_r}^*] = 0 \in H$.
\end{lemma}
\begin{proof}
The curve $(C,0) \subset (\C^2,0)$ defined by $f_0$ splits into branches
$C = \cup_{i,j} C_{i,j}$ where $C_{i,1} \cup \ldots \cup C_{i,s_i}$
correspond to the segment $[(a_{i-1},b_{i-1}), (a_i, b_i)]$.
Let $G_0$ be the graph associated with the minimal resolution
$V\to \C^2$ of $f_0$.
There are vertices $\bar n_i$ in $G_0$ so that the strict transforms
$\tilde C_{i,j}$ intersect the component $E_{\bar n_i}$ transversely
in one point each. The curve defined by $g_i$ is a curvette to $n_i$,
that is, if we define $D_i = \{g_i=0\} \subset \C^2$, then
the strict transform $\tilde D_i$ in the resolution of $C$ is smooth
and intersects $E_{\bar n_i}$ in one point, and is disjoint from
the $\tilde C_i$.

The resolution of $(X,0)$ is obtained by suspending the pull-back
of $f_0$ to $V$, resolving some cyclic qutotient singularities, and then
blowing down some $(-1)$-curves, see e.g. Appendinx 1 in
\cite{Nemethi_FL}. In particular, we have a morphism $\tilde X \to V$,
mapping $E_{n_i}$ to $E_{\bar n_i}$. The condition that $(X,0)$ has
a rational homology sphere link implies that this map is branched
of order $N$ along this divisor. As a result, it restricts to
an isomorphism $E_{n_i} \to E_{\bar n_i}$, and the preimage
$D_i$ of $\tilde C_i$ intersects $E_{n_i}$ transversally in one point.
Clearly, $D_i$ is the strict transform of the vanishing set of $g_i$
seen as a function on $X$. It follows that $\div_v g_i = E^*_{n_i}$.

Simlarly, one verifies that we have maps $E_{n_1} \to E_{\bar n_i}$,
which are branched covering maps of order $s_x$. Thus, the
strict transform of the vanishing set of $g_1$ in $X$ consists of
$s_x$ branches, each intersecting $E_{n_1}$ in one point. Thus,
$\div_v g_1 = s_x E^*_{n_1}$. A similar argument holds for $g_r$.
\end{proof}

\begin{definition}
Let $V'_\E = \Z\langle E^*_e | e\in\E\rangle$ and $V_\E = V'_\E \cap L$.
\end{definition}

The group $H = L'/L$ is generated
by residue classes of ends $[E_e^*]$, $e\in\E$. This is proved in
Proposition 5.1 of \cite{NeumWa_SpQ}.
In particular, the
natural morphism $V'_\E/V_\E \to H$ is an isomorphism.

\begin{lemma} \label{lem:generators}
The lattice $V_\E$ is generated by the following elements
\[
  N E_e^*,\,e\in\E^z,\quad \alpha_x s_x E_e^*,\,e\in\E^x,\quad
  \alpha_y s_y E_e^*,\,e\in\E^x,
\]
\[
  \alpha_x(E_{e_i^x}^* - E_{e^x_{i+1}}^*), \, 1\leq i< s_i,\quad
  \alpha_y(E_{e_i^y}^* - E_{e^y_{i+1}}^*), \, 1\leq i< s_y,
\]
\[
  \div(t) = \sum_{e\in\E^t} E_e^*,\quad t=x,y,z.
\]
\end{lemma}
\begin{proof}
We start by noting that
by \cref{lem:En_triv,lem:ends_nodes}, if $1<i<r$ and $e\in \E_{n_r}$, then
\[
  N E^*_e
  = \alpha_e E^*_e
  \equiv E_n^*
  \equiv 0 \quad (\mathrm{mod} \, L),
\]
i.e. $N E^*_e \in V_{\E}$. Similarly, if $e\in \E^z_1$, then
\[
  N E^*_e
  = s_x \alpha_e E^*_e
  \equiv s_x E_{n_1}^*
  \equiv 0 \quad (\mathrm{mod} \, L),
\]
and $N E^*_e \in V_{\E}$ for $e\in \E^z_r$ as well.
A similar argument shows $\alpha_x s_x E^*_e \in V_{\E}$ for
$e \in \E^x$ and $\alpha_y s_y E^*_e \in V_{\E}$ for $e\in \E^y$.
Let $A$ be the sublattice
of $V'_{\E}$ generated by these elements, that is, the top row
in the statement of the lemma. We then have $A \subset V_{\E}$,
and $[V'_{\E}:A] = (\alpha_x s_x)^{s_x} (\alpha_y s_y)^{s_y} N^{s_z}$.
By \cref{lem:H_order}, we get
\begin{equation} \label{eq:index}
  [V_{\E}:A]
  = [V'_{\E}:V_{\E}]^{-1} [V'_{\E}:A]
  = \alpha_x s_x^{s_x} \alpha_y s_y^{s_y} N.
\end{equation}

The elements in the second row are also elements of $V_\E$, since,
by \cref{lem:ends_nodes} we have
\[
  \alpha_x\left(E^*_{e_i^x} - E^*_{e_{i+1}^x}\right)
  =   \left(\alpha_xE^*_{e_i^x} - E^*_{n_1}\right)
    - \left(\alpha_xE^*_{e_{i+1}^x} - E^*_{n_1}\right)
  \in L,
\]
and similarly for $\alpha_y\left(E_{e_i^y}^* - E_{e^y_{i+1}}^*\right)$.
Let $A'$ be the subgroup of $V_{\E}$ generated by $A$ and these elements. Then
$[A':A] = s_x^{s_x-1} s_y^{s_y-1}$.

Finally, we have $\div(t) = \sum_{e\in \E^t} E^*_e \in L$
for $t=x,y,z$. Define $A''$ as the subgroup of $V_{\E}$
generated by $A'$ and $\div(t)$, $t=x,y,z$. Then
$[A'':A'] = (\alpha_x s_x) \cdot (\alpha_y s_y) \cdot N$,
and so $[A'':A] = \alpha_x s_x^{s_x} \alpha_y s_y^{s_y} N = [V_{\E}:A]$,
which gives $A'' = V_{\E}$.
\end{proof}

\begin{lemma} \label{lem:wt_f_z}
We have $\wth f|_\Nd = N \wth z|_\Nd$.
\end{lemma}
\begin{proof}
Indeed, every compact facet of $\Gamma_+(f)$ contains $(0,0,N)$.
\end{proof}

\begin{definition}[\cite{Nem_Poinc}] \label{def:zeta}
The \emph{zeta function}
associated with the graph $G$ is the expansion
at the origin of the rational function
$Z(t) = \prod_{v\in\V} \left( 1 - [E_v^*] t^{E_v^*} \right)^{\delta_v - 2}$.
Thus, if $G$ has more than one vertex, then we can write
\[
  Z(t) = \left[\prod_{n\in\Nd}
            \left( 1 - [E_v^*] t^{E_v^*} \right)^{\delta_v - 2}\right]
      \left[\prod_{e\in\E}
            \sum_{k=0}^\infty \left([E_e^*] t^{E^*_e}\right)^k \right]
   \in \Z[H][[t^{L'}]],
\]
whereas if $G$ has exactly one vertex, say $v$, then
\[
  Z(t) = (1-[E_v^*]t^{E^*_v})^{-2}
       = \sum_{k=0}^\infty (k+1) \left([E_v^*]t^{E^*_v}\right)^k.
\]
This latter case does not appear in our study of suspension singularities.
Here, $t$ denotes variables indexed by $\V$, and so if
$l = \sum_{v\in \V} l_v E_v \in L'$ with $l_v \in \Q$, then we write
$t^l = \prod_{v\in\V} t_v^{l_v}$.

We have $Z(t) \in \Z[H][[t^{L'}]] \cong \Z[[t^{L'}]][H]$, and the coefficient
in front of $t^l$ is in $[l]\cdot \Z \subset \Z[H]$. Therefore, we have
a decomposition $Z(t) = \sum_{h\in H} h\cdot Z_h(t)$ with
$Z_h(t) \in \Z[[t^{L'}]]$ for each $h\in H$.
In particular, $Z_0(t) \in \Z[[t^L]]$.

The \emph{reduced zeta function} $Z^\Nd(t)$ with respect to $\Nd$ is obtained
from $Z(t)$ by restricting $t_v = 1$ for $v\notin\Nd$.
By restricting $Z_0(t)$ similarly, we obtain
$Z_0^\Nd(t) \in \Z[[t^{\bar L}]]$.

In general, if $A(t) = \sum_{l\in L'} a_l t^l$ is a powerseries, then
we discard terms corresponding to $l \notin L$ by
setting $A_0(t) = \sum_{l\in L} a_l t^l$
\end{definition}

\begin{thm} \label{thm:zeta}
Assume that $G$ is the resolution of a Newton nondegenerate suspension
singularity, with rational homology sphere link. Then
\[
  Z_0^\Nd(t) = \frac{1-t^{\wth f}}{(1-t^{\wth x})(1-t^{\wth y})(1-t^{\wth z})},
\]
where, on the right hand side, we restrict to variables associated with
nodes only, i.e. we set $t_v = 1$ if $v\notin\Nd$.
\end{thm}
\begin{proof}
We assume that $s_x > 1$ and $s_y = 1$. The other cases are obtained by a
small variation of this proof. Note that in this case we have
$s_1 = 1$.

In what follows, we always assume all divisors to be restricted to $\Nd$.
In particular, in view of \ref{lem:ends_nodes}, we can make the identification
$([E^*_e] t^{E_e})^{\alpha_e} = [E^*_n] t^{E_n}$ for any $n\in\Nd$
and $e\in\E_n$. Given our assumption, we have $E_{e_1^y}^* = \wt y \in L$.
This means that if we write $Z'(t) = Z(t)(1-t^{E_{e_1^y}^*})$ we have
$Z_0(t) = Z'_0(t)/(1-t^{\wt y})$. We can therefore focus on $Z_0'$ instead
of $Z_0$. Write
\[
\begin{split}
  Z'(t)=& \frac{\left(1-[E^*_{n_1}] t^{E_{n_1}^*}\right)^{s_x}}
              {\prod_{i=1}^{s_x}\left(1-[E^*_{e_i^x}] t^{E_{e_i^x}^*}\right)}
          \cdot
		  \frac{1}{1-[E^*_{e_{1,1}^z}] t^{E_{e_i^x}^*}}
          \cdot
		  \prod_{i=2}^r
		  \frac{\left(1-[E^*_{n_i}] t^{E_{n_i}^*}\right)^{s_i}}
		  {\prod_{k_i=1}^{s_i} 1-[E^*_{e^z_{i,k_i}}] t^{E_{e^z_{i,k_i}}^*}}\\
       =& \prod_{i=1}^{s_x}
	      \sum_{j_i = 0}^{\alpha_x-1}
               \left([E^*_{e_i^x}] t^{E_{e_i^x}^*}\right)^{j_i}
		  \cdot
		  \sum_{l=0}^\infty \left([E^*_{e_{1,1}^z}] t^{E_{e_i^x}^*}\right)^l
          \cdot
		  \prod_{i=2}^r
          \prod_{k_i=1}^{s_i}
		  \sum_{l_{i,k_i}=0}^{N-1}
		  \left([E^*_{e^z_{i,k_i}}] t^{E_{e^z_{i,k_i}}^*}\right)^{l_{i,k_i}}
\end{split}
\]
Considering the presentation for $H$  given
in \ref{lem:generators}, one sees that if the coefficient
\[
  \prod_{i=1}^{s_x} [E^*_{e_i^x}]^{j_i} \cdot [E^*_{e_{1,1}^z}]^l
  \cdot \prod_{i=2}^r \prod_{k_i=1}^{s_i} [E^*_{e^z_{i,k_i}}]^{l_{k_i}}
  =
  \left[
  \sum_{i=1}^{s_x} j_iE^*_{e_i^x} + lE^*_{e_{1,1}^z}
  + \sum_{i=2}^r \sum_{k_i=1}^{s_i} l_{k_i}E^*_{e^z_{i,k_i}}
  \right]
\]
is trivial and $0\leq j_i < \alpha_x$, then in fact $j_i$ is constant
and both $\prod_{i=1}^{s_x} [E^*_{e_i^x}]^{j_i}$ and 
$[E^*_{e_{1,1}^z}]^l\cdot
\prod_{i=2}^r \prod_{k_i=1}^{s_i} [E^*_{e^z_{i,k_i}}]^{l_{k_i}}$ are trivial.
Therefore we get
\[
  Z_0'(t) = \left(\prod_{i=1}^{s_x}
	      \sum_{j_i = 0}^{\alpha_x-1} \left
               ([E^*_{e_i^x}] t^{E_{e_i^x}^*}\right)^{j_i}
		  \right)_0
		  \cdot
		  \left(
		  \frac{1}{1-[E^*_{e_{1,1}^z}] t^{E_{e_i^x}^*}}.
		  \prod_{i=2}^r
		  \prod_{k_i=1}^{s_i}
		  \sum_{l_{i,k_i}=0}^{N-1}
		    \left([E^*_{e^z_{i,k_i}}] t^{E_{e^z_{i,k_i}}^*}\right)^{k_i}
		  \right)_0
\]
and
\[
  \left(\prod_{i=1}^{s_x}
	      \sum_{j_i = 0}^{\alpha_x-1}
                \left([E^*_{e_i^x}] t^{E_{e_i^x}^*}\right)^{j_i}
		  \right)_0
  = \sum_{j=0}^{\alpha_x-1}
    t^{j\left(E_{e_1^x}^* + \cdots + E_{e_{s_x}^x}^*\right)}
  = \frac{1-t^{\alpha_x \wt x}}{1-t^{\wt x}}
\]
We have 
$t^{\alpha_x \wt x} = t^{s_x E_{n_1}^*}
= ([E^*_{e_{1,1}^z}] t^{E_{e_{1,1}^z}^*})^N$ by \ref{lem:ends_nodes}.
Thus, we may continue
\[
\begin{split}
  Z_0'(t) =&
		  \frac{1}{1-t^{\wt x}}
		  \cdot
		  \left(
		  \frac{1-\left([E^*_{e_{1,1}^z}] t^{E_{e_{1,1}^z}^*}\right)^N}
		       {1-[E^*_{e_{1,1}^z}] t^{E_{e_i^x}^*}}
		  \prod_{i=2}^r
		  \prod_{k_i=1}^{s_i}
		  \sum_{l_{i,k_i}=0}^{N-1}
		    \left([E^*_{e^z_{i,k_i}}] t^{E_{e^z_{i,k_i}}^*}\right)^{k_i}
		  \right)_0 \\
         =&
		  \frac{1}{1-t^{\wt x}}
		  \cdot
		  \left(
		  \prod_{i=1}^r
		  \prod_{k_i=1}^{s_i}
		  \sum_{l_{i,k_i}=0}^{N-1}
		    \left([E^*_{e^z_{i,k_i}}] t^{E_{e^z_{i,k_i}}^*}\right)^{l_{i,k_i}}
		  \right)_0.\\
\end{split}
\]
From \cref{lem:generators} one sees that 
$\prod_{i=1}^r \prod_{k_i=1}^{s_i}[E^*_{e^z_{i,k_i}}]^{l_{i,k_i}}$ is
trivial (assuming $0\leq l_{i,k_i} < N$) if and only if $l_{i,k_i}$
is constant. Thus, 
\[
		  \left(
		  \prod_{i=1}^r
		  \prod_{k_i=1}^{s_i}
		  \sum_{l_{i,k_i}=0}^{N-1}
		    \left(
              [E^*_{e^z_{i,k_i}}] t^{E_{e^z_{i,k_i}}^*}
            \right)^{l_{i,k_i}}
		  \right)_0
		  =
		  \sum_{l=0}^{N-1} 
          \left(
		  \prod_{i=1}^r
		  \prod_{k_i=1}^{s_i}
		    t^{E_{e^z_{i,k_i}}^*}\right)^l
          = \frac{1-t^{N \wth z}}{1-t^{\wth z}}.
\]
We therefore get, using \cref{lem:wt_f_z},
\[
  Z_0 = \frac{1}{1-t^{\wth y}}
        \cdot
        \frac{1}{1-t^{\wth x}}
        \cdot
        \frac{1-t^{\wth f}}{1-t^{\wth z}}
\]
which finishes the proof.
\end{proof}

\section{An example}
Let
\[
  f(x,y,z) = x^9 + x^4y^2 + x^2y^4 + y^7 + z^7.
\]
In this case we have $N = 7$ and by \cref{thm:susp}
\[
  s_x = 7,\quad \alpha_x = 2, \quad s_1 = 1, \quad
  s_y = 1,\quad \alpha_y = 2, \quad s_3 = 7, \quad
\]
\[
  s_z = s_1 + s_2 + s_3 = 1 + 2 + 1 = 4.
\]
By \cref{lem:H_order}, we have $|H| = 7^3 2^6 = 21952$, and
\[
  P^\F(t) = Z_0^\Nd(t)
  = \frac{1-t_1^{14} t_2^{42}t_3^{126}}
         {(1-t_1^{3} t_2^{7} t_3^{14})
          (1-t_1^{2} t_2^{7} t_3^{35})
          (1-t_1^{2} t_2^{6} t_3^{18})}.
\]
\begin{figure}[ht]
\begin{center}
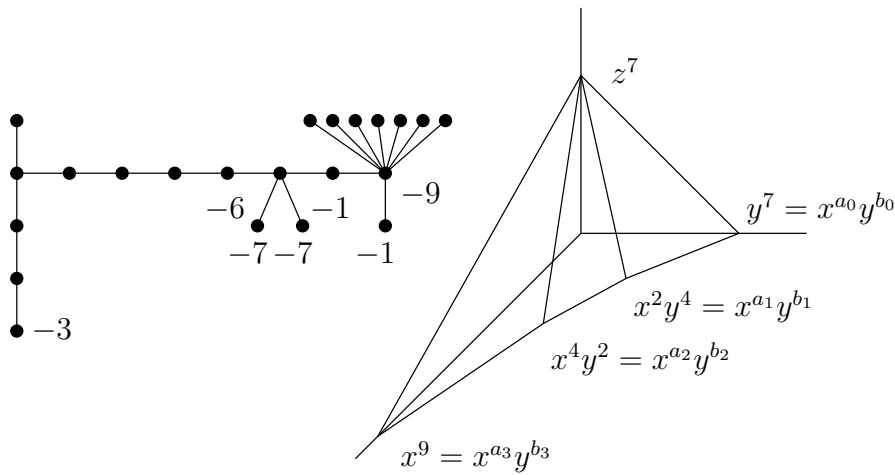
\caption{Unmarked vertices have Euler number $-2$.}
\label{fig:ex}
\end{center}
\end{figure}


\end{document}